\documentclass[DIV=13]{scrartcl}
\usepackage{amsmath,amsthm,upref, amssymb}
\usepackage[shortlabels]{enumitem}
\usepackage{mathtools, bbm}
\usepackage{microtype}
\usepackage{xcolor}
\usepackage[sortcites=true,giveninits=true,maxbibnames=10, backend=biber]{biblatex}
\AtEveryBibitem{\clearfield{issn}\clearfield{doi}\clearfield{urldate}}
\AtEveryCitekey{\clearfield{issn}\clearfield{doi}\clearfield{urldate}}
\renewrobustcmd*{\bibinitdelim}{\,}
\usepackage[unicode]{hyperref}
\hypersetup{
    colorlinks=true,
    linkcolor=blue,
    citecolor=red,
    filecolor=magenta,      
    urlcolor=cyan,
    pdftitle={ISS and integral ISS},
    linktocpage=true,
}
\DeclareSourcemap{
  \maps[datatype=bibtex]{
    \map[overwrite]{
      \step[fieldsource=shortjournal,fieldtarget=journaltitle]
    }
  }  
}
\theoremstyle{definition}
\newtheorem{definition}{Definition}[section]
\newtheorem{example}[definition]{Example}

\newtheorem{remarks}[definition]{Remarks}
\newtheorem{corollary}[definition]{Corollary}
\theoremstyle{plain}
\newtheorem{theorem}[definition]{Theorem}
\newtheorem{lemma}[definition]{Lemma}
\newtheorem{proposition}[definition]{Proposition}
\newcommand{\C}{\mathbb{C}}
\newcommand{\R}{\mathbb{R}}
\newcommand{\N}{\mathbb{N}}

\newcommand{\Ll}{\mathrm{L}}
\newcommand{\Cc}{\mathrm{C}}
\newcommand{\LL}{\mathcal{L}}
\newcommand{\Ee}{\mathrm{E}}
\newcommand{\e}{\mathrm{e}}
\newcommand{\dd}{\,\mathrm{d}}
\newcommand{\ii}{\mathrm{i}}
\newcommand{\Abs}[1]{\!\left| #1 \right|}
\newcommand{\abs}[1]{| #1 |}
\newcommand{\Norm}[1]{\!\left\lVert #1 \right\rVert}
\newcommand{\norm}[1]{\lVert #1 \rVert}
\newcommand{\Intnorm}[1]{\biggl\lVert #1 \biggr\rVert}

\newcommand{\Spr}[2]{\left\langle #1 , #2 \right\rangle}
\newcommand{\Tt}{T(t)}
\newcommand{\Ttt}{$(\Tt)_{t\ge 0}$}
\renewcommand{\epsilon}{\varepsilon}
\renewcommand{\phi}{\varphi}

\DeclareMathOperator{\vecspan}{span}

\newcommand{\ball}{\mathbf{B}}
\newcommand{\czero}{$\textup{C}_0$}

\newcommand{\calK}{\mathcal{K}}

\newcommand{\argument}{\mathord{\,\cdot\,}} 

\newcommand{\resSet}{\rho}
\newcommand{\Res}{\mathcal{R}} 
\newcommand{\dom}[1]{\operatorname{dom}\left(#1\right)} 
\DeclareMathOperator{\one}{{\mathbbm{1}}} 
\newcommand\restrict[1]{\raisebox{-.5ex}{$|$}_{#1}} 
\newcommand{\sun}{\odot} 
\newcommand{\mirrorR }{ \text{\reflectbox{\( \mathbf R \)}}} 
\DeclareMathOperator*{\esssup}{ess\,sup}

\setlist{itemsep=0.2ex, parsep=0.4ex}
\date{September 30, 2026}
\author{Sahiba Arora\thanks{Institut für Analysis, Leibniz Universität Hannover, Welfengarten 1, 30167 Hannover, Germany. E-mail: \texttt{sahiba.arora@math.uni-hannover.de}}\ , Philip Preußler\thanks{Mathematics of Systems Theory, Department of Applied Mathematics, University of Twente, P.\,O.\ Box 217, 7500~AE  Enschede, The Netherlands. E-mail: \texttt{p.n.preusler@utwente.nl}} \ and Felix L.\ Schwenninger\thanks{Mathematics of Systems Theory, Department of Applied Mathematics, University of Twente, P.\,O.\ Box 217, 7500~AE  Enschede, The Netherlands. E-mail: \texttt{f.l.schwenninger@utwente.nl}}}
\title{ISS and integral~ISS coincide for linear systems with bounded inputs}
\begin{document}
\maketitle
\begin{abstract}
    \noindent\textbf{Abstract.}
    We show that input-to-state stability (ISS) and integral input-to-state stability (integral~ISS)
    with respect to the space of essentially bounded functions are equivalent properties for linear systems. The proof combines recent results on Orlicz-admissibility and a duality result due to the authors. As a direct consequence, we show that, given an $L^\infty$-admissible control operator, the mild solutions of the linear system are continuous, resolving a question by G.\ Weiss.
\end{abstract}\vspace{3ex}
\vspace{-2ex}
\section{Introduction}

\emph{Input-to-state stability (ISS)} is a key concept to jointly study internal stability and robustness to external forcing terms for large classes of systems, given by a control/initial-state-to-state relation, 
\begin{equation*}
 u,x(t_{0})\mapsto x(T), 
\end{equation*}
where $u\colon[t_{0},T]\to U$ and $x\colon[t_{0},T]\to X$ are respectively $U$-valued and $X$-valued signals in time $t\in[t_{0},T]$. Introduced by Sontag \cite{sontagSmoothStabilization1989}, it has flourished as its own field within control theory, both due to its theoretical strength and applicability in terms of practical assessment of the property, e.g.\ Lyapunov type theorems. Most importantly, it supplies a framework that allows to bridge formerly independent properties  as Lyapunov stability (of dynamical systems) and robustness of external inputs or disturbances. Especially for nonlinear systems, such a unified treatment is far from trivial; see \cite{sontagISSBasicConcepts2008} for an overview.

Yet, the history of ISS and its variants \cite{sontagCommentsIntegralVariants1998}, was initially limited to finite-dimensional systems,   
\begin{equation}
    \dot{x}(t)=f(x(t),u(t)), \qquad x(0)=x_{0};
\end{equation}
with $X$ and $U$ finite-dimensional and functions $f\colon X\times U\to X$  chosen such that the system is well posed
and every initial state and admissible input give rise to a
unique global solution (\emph{forward completeness}). In this situation, ISS is defined by saying that functions $\beta\in\mathcal{KL}$ and $\gamma\in \mathcal{K}$ in the common Lyapunov classes exist such that 
\begin{equation}\label{eq:iss}\tag{ISS}
    \|x(T)\|_X\leq \beta(\|x_{0}\|_X,T)+\gamma(\esssup_{t\in[0,T]}\|u(t)\|_U),
\end{equation}
for all times $T>0$,  initial states $x_{0}$ and inputs $u$. The \emph{linear}, finite-dimensional subcase turns out to be rather elementary, by the variation-of-constants formula, as ISS, with respect to all $x_{0}\in X = \R^d$ and continuous, regulated, or (essentially) bounded $u$, readily reduces to uniform asymptotic stability of the internal dynamics, or, in other words, that the matrix governing the state dynamics is Hurwitz. In this linear case, $\beta(s,t)=Ms\mathrm{e}^{-t\omega}$, and $\gamma(r)=Cr$ with some  $C,M,\omega>0$ can be chosen.
 
A specific variant of ISS, which arose from studying bilinear equations, is \emph{integral input-to-state stability (integral~ISS)}, in which the supremum norm of $u$ in \eqref{eq:iss} is replaced by an ``integral-type term'', i.e.,
\begin{equation*}
    \label{eq:integraliss}\tag{integral~ISS}
    \|x(T)\|_X\leq \beta(\|x_{0}\|_X,T)+\gamma\left(\int_{0}^{T}\alpha\left(\|u(t)\|_U\right)\dd t\right),
\end{equation*}
for all $u \in \Ll^\infty([0,\infty);U)$ and with $\alpha\in\mathcal{K}$. From the internal stability characterization, it is clear that for linear, finite-dimensional systems, integral~ISS is also equivalent to ISS.

Somewhat surprisingly, this simplicity does not carry over to linear, \emph{infinite-dimensional} systems on Banach spaces $X$ and $U$, as given by
\begin{equation}
    \label{eq:control-system}
    \dot x (t) = Ax(t) + Bu(t),\quad t\ge 0, \qquad x(0)=x_{0}\in X; 
\end{equation}
where $A$ is supposed to generate a strongly continuous semigroup \Ttt{} on $X$
and $B\in\LL(U,X_{-1})$ for a Banach
space $U$. Fix $\lambda_0\in\rho(A)$. We write $X_1=\dom A$
equipped with the usual graph norm
and let $X_{-1}$ denote the completion of $X$ with respect to
$\norm{x}_{X_{-1}}:=\norm{\Res(\lambda_0,A)x}_X$.
The semigroup \Ttt{} extends to a strongly continuous
semigroup $(T_{-1}(t))_{t\geq0}$ on $X_{-1}$, whose generator
$A_{-1}$ has domain $X$.
The input $u\in\mathrm Z([0,t];U)$ belongs to a specified
function space of strongly measurable $U$-valued functions.
The reason for more general input operators\footnote{By convention, operators $B \in \LL(U,X_{-1})$ are called \emph{unbounded}, as they need not be bounded $U \to X$.} $B$  than bounded ones from $U$ to $X$ is to allow for more general control settings, such as \emph{boundary control}, where the input $u$ only enters via a spatial boundary (trace), see, e.g., \cite{lasiecka_control_vol1_2000,lasiecka_control_vol2_2000, staffans_well-posed_2005,tucsnak_observation_2009}.

While the variation-of-constants formula,
\begin{equation}\label{eq:mildsol}
 x(t)=T(t)x_{0}+\int_{0}^{t}T_{-1}(t-s)Bu(s)\,\mathrm{d}s,
\end{equation}
defines (mild) solutions $x\colon[0,\infty)\to X_{-1}$, for general unbounded $B$, it is not in general true that $x$ maps into $X$, which is required to even consider an ISS-type estimate. For this reason, the definition of (integral) ISS with respect to a function class $\mathrm Z$ for such infinite-dimensional systems described by $A$ and $B$ consists of both the condition that for all $x_{0}\in X$, $u\in \mathrm Z(0,t_{1};U)$ and all $t_{1} > 0$, the mild solution \eqref{eq:mildsol} gives an $X$-valued function $x$ and that Inequality \eqref{eq:iss}, respectively, Inequality \eqref{eq:integraliss}, holds, for suitable Lyapunov-class functions $\beta$, $\gamma$ (and $\alpha$); see Section~\ref{sec:main-results} below and \cite{jacobOnInputToStateStability2016}.  See also \cite{MironchenkoPrieur2020} for a recent survey on ISS in the context of infinite-dimensional systems; also, see the books \cite{mironchenkoISS2023,karafyllisISS_PDEs2019} for more on ISS.

 We note that the natural choices for $\mathrm Z$ are the ($U$-valued) continuous functions $\mathrm{C}$, the regulated functions $\mathrm{Reg}$ and the essentially bounded functions $\mathrm{L}^{\infty}$. 
 
 Since the early contributions \cite{mazencStrictLyapunov2011,DashkovskyMironchenko}, ISS and integral~ISS have been studied for various classes of infinite-dimensional systems. Mironchenko and Wirth \cite{MironchenkoWirth2018} established characterizations of ISS for a class of infinite-dimensional control systems and identified limitations of such extensions. For bilinear systems with bounded input operators, Mironchenko and Ito \cite{MironchenkoItoBilinear2016} proved that uniform global asymptotic stability is equivalent to integral~ISS and developed Lyapunov methods for ISS and integral~ISS of nonlinear parabolic equations and small-gain criteria for their interconnections \cite{MironchenkoItoParabolic2015}. Zheng and Zhu \cite{ZhengZhu2020} established ISS estimates for a class of one-dimensional nonlinear parabolic equations with boundary disturbances and nonlinear boundary conditions, using maximum estimates and Lyapunov methods. Guiver, Logemann and Opmeer \cite{GuiverLogemannOpmeer2019} obtained sufficient conditions for (incremental) ISS of infinite-dimensional Lur'e systems.

For the linear systems considered here, the underlying question is whether $\mathrm Z$-ISS is equivalent to $\mathrm Z$-integral~ISS; see \cite{Jacob_2018}. As mentioned above, this equivalence is elementary in finite dimensions.

 First note that, for linear systems, it is easy to show that $\mathrm Z$-integral~ISS always implies $\mathrm Z$-ISS, as the latter is equivalent to uniform asymptotic stability together with $\mathrm Z$-admissibility, \cite{Jacob_2018}. At the same time, it was shown that also $\mathrm Z$-integral~ISS, with $\mathrm Z$ being either $\mathrm{C}$, $\mathrm{Reg}$ or $\mathrm{L}^{\infty}$,  can be characterized by uniform asymptotic stability together with admissibility with respect to a weaker norm than the supremum norm -- more precisely, by admissibility with respect to an Orlicz-space norm. With this result at hand, several results on equivalence of ISS and integral~ISS have been shown in the past ten years, all making extra assumptions on either the $A$ or the $B$ operators; see \cite{Jacob_2018,jacobContinuitySolutionsParabolic2019}. Recently, structural results, only making extra assumptions on the Banach spaces $X$ and $U$ (and their geometry), were shown in \cite{preusslerImplicationsOfStructured2026}. In particular, it was shown that ISS and integral~ISS are equivalent for all of the above choices of $\mathrm Z$ if $X$ is a reflexive space. The subtlety of the question and these results becomes apparent from the fact that the equivalence fails for $\mathrm Z=\mathrm{C}$ and $\mathrm Z=\mathrm{Reg}$ in general. Indeed, as first shown by a classical counterexample due to Kato in \cite{JacobSchwenningerWintermayr2022}, $\mathrm C$-ISS does not imply $\mathrm C$-integral~ISS.

\subsection{Main results}
    \label{sec:main-results}

Let $A$ generate a \czero-semigroup \Ttt{} on a Banach space $X$ and let $B\in \LL(U,X_{-1})$ for some Banach space $U$. 
Let $\mathrm Z$ be a placeholder for $\Ll^\infty$ or $\Ee_F$ for some Young function $F$ (see Section~\ref{sec:admissibility} for the definition).
We say that $B$ is \emph{$\mathrm{Z}$-admissible} if there is $C > 0$ such that
\[
    \Intnorm{\int_0^1 T_{-1}(1-s) Bu(s) \dd s}_{X} \leq C \norm{u}_{\mathrm{Z}([0,1];U)}
    \qquad
    \text{for all }
    u \in \mathrm{Z}([0,1];U).
\]

In what follows, as is customary when discussing comparison function spaces, we denote by $\calK$ the set of all continuous and strictly increasing
functions $\gamma\colon\R_+\to\R_+$ satisfying $\gamma(0)=0$, and by
$\calK_\infty$ the subset of all unbounded functions in $\calK$.
We write $\mathcal{KL}$ for the continuous functions
$\beta\colon\R_+\times\R_+\to\R_+$ such that
$\beta(\argument,t)\in\calK$ for every $t\geq0$ and
$\beta(s,\argument)$ is decreasing and converges to zero
as $t\to\infty$ for every $s\geq0$.

The system $(A,B)$ is called \emph{input-to-state stable with respect to $\mathrm Z$} (for short, $\mathrm Z$-ISS) if every initial state and input give rise to an $X$-valued
mild solution and there exist $M,\omega>0$ and $\theta\in \calK_\infty$ such that the mild solution  satisfies
\[
    \norm{x(t)}_X  \le M e^{-\omega t}\norm{x_0}_X + \theta( \norm{u}_{\mathrm Z([0,t];U)})
\]
for all $t\ge 0$, $x_0\in X$, and $u\in \mathrm Z([0,t];U)$.
The linear system $(A,B)$ is called
\emph{integral input-to-state stable with respect to
$\Ll^\infty$} (for short, $\Ll^\infty$-integral~ISS) if every initial
state and essentially bounded input give rise to an
$X$-valued mild solution and there exist $M$, $\omega>0$,
$\theta\in\calK_\infty$, and $\mu\in\calK$ so that
\[
    \norm{x(t)}_X
    \leq Me^{-\omega t}\norm{x_0}_X
    +\theta\left(
        \int_0^t\mu\left(\norm{u(s)}_U\right)\dd s
    \right)
\]
for all $t\geq0$, $x_0\in X$, and
$u\in\Ll^\infty([0,t];U)$.
A related property of integral-to-integral ISS has recently been studied in \cite{AroraMironchenko2026}.
The following is the main result of this article and it settles \cite[Open Problem~3.22]{MironchenkoPrieur2020} for linear infinite-dimensional systems. Its proof and of the corollaries below are given in Section~\ref{sec:proofs}.

\begin{theorem}
    \label{thm:main}
    Let $X$ and $U$ be Banach spaces, let $A$ be the generator of a \czero-semigroup \Ttt{} on $X$,
    and let $B \in \LL(U,X_{-1})$ be arbitrary. Consider the following conditions:
    \begin{enumerate}[\upshape (i)]
        \item \label{thm:main:itm:infty-admissibility}
        $B$ is $\Ll^\infty$-admissible;
        \item \label{thm:main:itm:orlicz-admissibility}
        $B$ is $\Ee_F$-admissible for some Young function $F$;
        \item \label{thm:main:itm:iss}
        The system $(A,B)$ is $\Ll^\infty$-ISS;
        \item \label{thm:main:itm:iiss}
        The system $(A,B)$ is $\Ll^\infty$-integral~ISS.
    \end{enumerate}
    We always have \ref{thm:main:itm:iiss}$\Rightarrow$\ref{thm:main:itm:iss}$\Rightarrow$\ref{thm:main:itm:infty-admissibility}$\Rightarrow$\ref{thm:main:itm:orlicz-admissibility}.
    In addition, if \Ttt{} is exponentially stable, then~\ref{thm:main:itm:infty-admissibility}--\ref{thm:main:itm:iiss} are even equivalent.
\end{theorem}

For $1\leq p<\infty$, $\Ll^p$-admissibility of a control operator
implies continuity of the corresponding mild solutions as
$X$-valued functions \cite[Proposition~2.3]{weiss_admissibility_1989}. Whether the same remains true for $p=\infty$ is a long-standing problem posed by Weiss in \cite[Problem~2.4]{weiss_admissibility_1989}. Subsequent results established continuity under additional hypotheses, for instance, under certain assumptions on analytic semigroups \cite[Theorem~1]{jacobContinuitySolutionsParabolic2019}. Theorem~\ref{thm:main} gives an answer to the question without any additional hypothesis:

\begin{corollary}
    \label{cor:weiss}
    Let $X$ and $U$ be Banach spaces, let $A$ be the generator of a \czero-semigroup \Ttt{} on $X$. If $B\in \LL(U,X_{-1})$ is $\Ll^\infty$-admissible, then for each $x_0\in X$ and $u\in \Ll^\infty(\R_+; U)$, the mild solution $x$ given by~\eqref{eq:mildsol} of the control system~\eqref{eq:control-system} is a continuous function $x\colon [0,\infty) \to X$.
\end{corollary}

Admissible control operators also arise when the state is fed back as
the input. In the case $U=X$, this leads formally to the autonomous
equation $\dot x=A_{-1}x+Bx$. A natural question is whether the part of
$A_{-1}+B$ in $X$ generates a $C_0$-semigroup, so that the resulting
dynamics are well posed on $X$. The question is addressed by the bounded perturbation theorem
and the Desch--Schappacher perturbation theory
\cite[Sections~III.1 and~III.3.a]{engel_one-parameter_2000};
see also \cite{adlerOnPerturbations2014}. In particular, an affirmative answer is available when $B$ is $\Ll^p$-admissible, for $1\le p<\infty$. As another consequence of Theorem~\ref{thm:main}, we are able to
establish a generation result at the endpoint $p=\infty$:

\begin{corollary}
    \label{cor:perturbation}
    Let $A$ be the generator of a \czero-semigroup \Ttt{} on a Banach space $X$. If $B\in \LL(X,X_{-1})$ is $\Ll^\infty$-admissible, then the part of  $A_{-1}+B$ in $X$ generates a $C_0$-semigroup on $X$.
\end{corollary}

Theorem~\ref{thm:main} also has consequences for the study of bilinear systems.

\begin{corollary}
    \label{cor:bilinear}
    Let $A$ generate a \czero-semigroup on a Banach space $X$
    and let $B\in\LL(X,X_{-1})$.
    If the linear system $(A,B)$ is $\Ll^\infty$-ISS,
    then the bilinear system
    \[
        \dot x(t)=Ax(t)+u(t)Bx(t),
        \qquad x(0)=x_0\in X,
    \]
    is $\Ll^\infty$-integral~ISS: for every $x_0\in X$
    and $u\in\Ll^\infty(\R_+;\C)$, its mild solution
    exists globally in $X$ and satisfies
    \eqref{eq:integraliss} with $U=\C$ and suitable
    comparison functions $\beta$, $\gamma$, $\alpha$
    independent of $x_0$ and $u$.
\end{corollary}

More general versions of this result for system classes as in \cite{HosfeldJacobSchwenninger2022} can be shown analogously. E.g., consider the more general bilinear system class 
\[
    \dot x = Ax + B_1 F(x, u_1) + B_2 u_2,\quad x_0 \in X.
\]
Under the hypotheses on $F$ in
\cite[Section~2.1]{HosfeldJacobSchwenninger2022},
$\Ll^\infty$-integral~ISS follows if
 $(A,B)$ is
$\Ll^\infty$-ISS; compare also the system class in
\cite{HastirHosfeldSchwenningerWierzba2024}.
We refer to \cite{hosfeldISSbilinear2024} for more on ISS of bilinear systems.

\section{Lifting \texorpdfstring{{\boldmath$\Ll^\infty$}}{L^infty}-admissibility to Orlicz-admissibility}\label{sec:admissibility}

While $\Ll^\infty$-admissibility controls the response to bounded inputs, one can ask whether the input operator extends to a larger function space. A natural candidate is the \emph{Orlicz heart} $\Ee_F$ for an appropriate Young function $F$. This was posed as \cite[Question~24]{jacobContinuitySolutionsParabolic2019}. This question has been addressed for various classes of control systems, such as diagonal systems
\cite[Corollary~3.7 and Remark~3.9(3)]{JacobPartingtonPottRydheSchwenninger2026} and systems for which weak compactness or geometric assumptions on the state space are available \cite[Corollary~3.8 and Theorems~3.7 and~3.17]{preusslerImplicationsOfStructured2026}. In this section, we affirmatively answer the question without any additional assumptions on the control system.

Let $F\colon \R_+\to \R_+$ be a \emph{Young function}, i.e., $F$ is a convex, continuous, and increasing function such that $\lim_{t\to 0} t^{-1} F(t)=0$ and $\lim_{t\to \infty} t^{-1} F(t)=\infty$. Recall that for a Banach space $U$, the \emph{small Orlicz space} $\Ee_F([0,1];U)$ is defined as the completion of the bounded $U$-valued simple functions in the Luxemburg norm
\[
    \norm{u}_{\Ee_F([0,1];U)} \coloneqq \inf\left\{ \lambda>0 : \int_0^1 F\left( \lambda^{-1} \norm{u(s)}_U \right)\dd s\le 1\right\}.
\]

\begin{proposition}\label{prop:linfty_to_ef}
    Let $X$ and $U$ be Banach spaces and
    let $(T(t))_{t\ge 0}$ be a \czero-semigroup on $X$. If $B\in \LL(U,X_{-1})$ is $\Ll^\infty$-admissible, then $B$ is $\Ee_F$-admissible for some Young function $F$.
\end{proposition}

\begin{remarks}
    \label{rem:linfty-to-orlicz}
    (a) Note that the Young function $F$ in Proposition~\ref{prop:linfty_to_ef} depends on the control operator $B$. In particular, it generally cannot be replaced by $F(t)=t^p$ for finite $p$ since $\Ll^\infty$-admissibility need not imply $\Ll^p$-admissibility for any finite $p$ \cite[Example~5.2]{Jacob_2018}.

    (b) An important consequence of Proposition~\ref{prop:linfty_to_ef} is that admissibility of an $\Ll^\infty$-admissible control operator is automatically zero-class, i.e., 
    $
        \lim_{t\downarrow0}\Norm{u\mapsto\int_0^t T_{-1}(t-s) Bu(s) \dd s}_{\LL(\Ll^\infty([0,t];U),X)  }  =0;
    $
    see \cite[Proof of Proposition~5]{jacobContinuitySolutionsParabolic2019}.

    In particular, Proposition~\ref{prop:linfty_to_ef} provides an alternate proof of the implication (1)$\Rightarrow$(3) in \cite[Theorem~2.9]{JacobSchwenningerWintermayr2022}, i.e., if $A_{-1}$ is $\Ll^\infty$-admissible, then $A$ is bounded (cf. \cite[Proposition~1.3]{JacobSchwenningerWintermayr2022}). 

    (c) Under the remaining assumptions
    on the nonlinearity, the local existence and well-posedness results
    in \cite[Theorems~3.7 and~3.17]{Mironchenko2024} apply whenever
    both control operators are $\Ll^\infty$-admissible:~the continuity requirement in
    \cite[Assumption~3.1]{Mironchenko2024} follows from
    $\Ll^\infty$-admissibility of $B$, while the zero-class
    $\Ll^\infty$-admissibility requirement on $B_2$ in
    \cite[Assumption~3.2]{Mironchenko2024} can be replaced by
    $\Ll^\infty$-admissibility as noted in (b) above.
\end{remarks}

An important ingredient of the proof of Proposition~\ref{prop:linfty_to_ef} is the fact that the input operator corresponding to a $\Ll^\infty$-admissible control operator with finite-dimensional input space maps into a separable subspace. We outsource this argument to the following lemma.

\begin{lemma}
    \label{lem:weak-compactness-input-operator}
    Let $X$ and $U$ be Banach spaces,
    let $(T(t))_{t\ge 0}$ be a \czero-semigroup on $X$, and let $B\in \LL(U,X_{-1})$ be a $\Ll^\infty$-admissible control operator with input operator
    \[
        \Phi_1\colon \Ll^\infty([0,1];U) \to X,
        \qquad
        u \mapsto \int_0^1 T_{-1}(1-s) Bu(s) \dd s.
    \]
    Then for each $v\in \Ll^\infty([0,1];U)$, the operator
    \[
        \Phi_{1,v}\colon \Ll^\infty([0,1]) \to X, \qquad f \mapsto \Phi_1(fv)
    \]
     is weakly compact. 
     In particular, it follows that the input operators corresponding to $\Ll^\infty$-admissible input elements $b\in X_{-1}$ are weakly compact.
\end{lemma}

\begin{proof}
    Let $v \in \Ll^\infty([0,1];U)$ be arbitrary.
    Since $v\in \Ll^\infty([0,1];U)$ is strongly measurable, there exists a separable subspace $V\subseteq U$ such that $v(s)\in V$ for almost every $s$.
    This follows from the Pettis measurability theorem; see, e.g., \cite[Theorem II.1.2]{DiestelUhl1977}. 
    
    Fix $\lambda\in \resSet(A)$ and define the set
    \[
        Z \coloneqq \overline{\vecspan} \{ T(t)\Res(\lambda,A_{-1})Bz : t\ge 0, z\in V\};
    \]
    note that $Z$ is separable, closed, and $(T(t))_{t\ge 0}$-invariant. Here, the separability follows from 
    separability of $V$ and an argument using the density of the rationals and the strong continuity of the semigroup \Ttt. Moreover, on setting
    \begin{align*}
        z \coloneqq \Res(\lambda ,A) \Phi_1 v
          = \int_0^1 T(1-s)\Res(\lambda,A_{-1}) B v(s)\dd s,
    \end{align*}
    we see that $z \in Z \cap \dom A$.
    In particular, $z$ and $Az$ are elements of $Z$. This yields $\Phi_1 v =(\lambda-A)z\in Z$.
    Consequently, $\Phi_{1,v}(\Ll^\infty([0,1])) \subseteq Z$.
    Now, $\Ll^\infty([0,1])$ is Grothendieck and $Z$ is separable. Weak compactness of $\Phi_{1,v}\colon \Ll^\infty([0,1]) \to Z$ follows by  \cite[Corollary~VI.2.12]{DiestelUhl1977}. Thus, $\Phi_{1,v}\colon \Ll^\infty([0,1]) \to X$ is also weakly compact.
    Finally, if $b\in X_{-1}$ is $\Ll^\infty$-admissible, then taking $U= \C$ 
    and $v\equiv \one$ above, the corresponding input operator is 
    $\Phi_{1,\one}$, hence weakly compact.
\end{proof}

In the sequel, we use the notation $\ball_X$ for the closed unit ball of a Banach space $X$. Moreover, we use $\abs{E}$ for the Lebesgue measure of a set $E$.

The proof of Proposition~\ref{prop:linfty_to_ef} uses the following criterion for uniform integrability. Recall that a set $M\subseteq \Ll^1([0,1];U)$, where $U$ is a Banach space, is called \emph{uniformly integrable} if for each sequence of measurable subsets $(E_n)_{n\in \N}$ of $[0,1]$ satisfying $\Abs{E_n}\to 0$, one has
\[
    \lim_{n\to \infty}\sup_{f\in M} \int_{E_n} \norm{f(s)}_U\dd s = 0.
\]
For uniform integrability and its relation to weak compactness,
see \cite[Section~IV.2]{DiestelUhl1977}.

\begin{lemma}
    \label{lem:uniform-integrability}
    Let $U$ be a Banach space and let $M\subseteq \Ll^1([0,1];U')$ be norm bounded. If for each $v\in \Ll^\infty([0,1];U)$, the set
    \[
        M_v \coloneqq \{ \Spr{f(\argument)}{v(\argument)}_{U',U} : f\in M\} \subset \Ll^1([0,1])
    \]
    is uniformly integrable, then $M$ is a uniformly integrable subset of $\Ll^1([0,1];U')$.
\end{lemma}

\begin{proof}
    We argue the contrapositive.
    Suppose $M$ is not uniformly integrable; then because $M$ is bounded, there exists $\epsilon>0$, a sequence $(f_n)_{n\in\N}\subseteq M$ and pairwise disjoint measurable sets $E_n\subseteq [0,1]$ such that $\int_{E_n} \norm{f_n(s)}_{U'}\dd s\ge \epsilon$ for all $n\in\N$; see \cite[proof of Theorem~IV.2.4]{DiestelUhl1977}.

    Due to isometry of the canonical embedding $\Ll^1([0,1];U')$ into $\Cc([0,1];U)'$, the function $h_n \coloneq \one_{E_n} f_n \in \Ll^1([0,1];U')$ satisfies
    \[
        \norm{h_n}_{\Ll^1([0,1];U')} = \sup \left\{  \Abs{\int_0^1 \Spr{h_n(s)}{v(s)}_{U',U}\dd s } : v\in  \ball_{\Cc([0,1];U)}  \right\}.
    \]
    In particular,
    \[
         \sup \left\{  \Abs{\int_{E_n}\Spr{h_n(s)}{v(s)}_{U',U}\dd s } : v\in  \ball_{\Cc([0,1];U)} \right\} = \int_{E_n} \norm{f_n(s)}\dd s \ge \epsilon.
    \]
    We can thus choose a sequence $(v_n)_{n\in \N} \subseteq \Cc([0,1];U)$ such that $\norm{v_n}_{\Cc([0,1];U)} \le 1$ and \[\Abs{\int_{E_n}\Spr{h_n(s)}{v_n(s)}_{U',U}\dd s }\ge \epsilon/2\] hold for every $n\in \N$.
    The fact that the sets $E_n$ are mutually disjoint ensures that $v\coloneq \sum_n \one_{E_n} v_n$ lies in $\Ll^\infty([0,1];U)$.
    Since $h_n=f_n$ and $v=v_n$ on $E_n$,
    \[
        \int_{E_n}\Abs{\Spr{f_n(s)}{v(s)}_{U',U}}\dd s
        \geq
        \Abs{\int_{E_n}\Spr{h_n(s)}{v_n(s)}_{U',U}\dd s}
        \geq\frac{\epsilon}{2}.
    \]
    Disjointness gives $\sum_n\abs{E_n}\leq1$, hence
    $\abs{E_n}\to0$. Thus, $M_v$ is not uniformly integrable.
\end{proof}

Finally, the proof of Proposition~\ref{prop:linfty_to_ef} uses the sun dual theory of \czero-semigroups. Let $(T(t))_{t\ge 0}$ be a \czero-semigroup on a Banach space $X$. The associated sun-dual space is the subspace $X^\sun$ of $X'$ where the dual semigroup $(T(t)')_{t\ge 0}$ is strongly continuous. The restricted \czero-semigroup is denoted by $(T^\sun(t))_{t\ge 0}$. For a comprehensive overview of the theory, see \cite{vanNeerven1992}.
It is well-known that the canonical identification between $(X^\sun)_1$ and $(X_{-1})^\sun$ is an isomorphism. In particular, a control operator $B\in \LL(U, X_{-1})$ naturally gives rise to an observation operator $B'\restrict{(X^\sun)_1}\colon (X^\sun)_1 \to U'$.
 The connections between admissibility and sun-duality were recently studied in \cite{aroraAdmissibleOperatorsSundual2026}.

\begin{proof}[Proof of Proposition~\ref{prop:linfty_to_ef}]
    We show $\Ee_F$-admissibility using the uniform integrability approach from \cite{preusslerImplicationsOfStructured2026}. Let $\Phi_1\colon u \mapsto \int_0^1 T_{-1}(1-s) Bu(s) \dd s$ denote the
    input operator corresponding to $B$, which is bounded $\Ll^\infty([0,1];U) \to X$ by assumption.

    Since $B$ is $\Ll^\infty$-admissible for $A$, we know from \cite[Theorem~3.2]{aroraAdmissibleOperatorsSundual2026} that $B^\sun \coloneqq B'\restrict{(X^\sun)_1}$ is an $\Ll^1$-admissible observation operator for the sun dual semigroup $(T^\sun(t))_{t\ge 0}$ on $X^\sun$. 
    We denote the extension of the corresponding output operator $\Psi_1 \colon (X^\sun)_1 \to \Ll^1([0,1]; U')$, $x^\sun \mapsto B^\sun T^\sun (\argument) x^\sun$ to $X^\sun$ also by $\Psi_1$. 
    Consider the reflection operator $\mirrorR_1\in \LL(\Ll^1([0,1];U'))$ given by $f \mapsto f(1-\argument)$. 
    For each $x^\sun \in (X^\sun)_1$ and $u\in \Ll^\infty([0,1],U)$, we have
    \begin{align*}
        \Spr{x^\sun}{\Phi_{1}u} & = \int_{0}^{1} \Spr{x^\sun}{T_{-1}(1-s)Bu(s)  }_{(X_{-1})',X_{-1}}\dd s
                                     = \int_{0}^1 \Spr{B^\sun (T^{\sun})_1(1-s)  x^\sun}{u(s)  }_{U',U}\dd s\\
                                    & = \Spr{ \mirrorR_1\Psi_1 x^\sun }{u}_{ \Ll^\infty([0,1],U)', \Ll^\infty([0,1],U)};
    \end{align*}
     where the last equality uses that $\Ll^1([0,1],U')$ sits inside the dual of $\Ll^\infty([0,1],U)$. The pairing with the extrapolation-space integral is
    legitimate because $x^\sun\in(X^\sun)_1$ represents a
    bounded functional on $X_{-1}$.
     By density of $(X^\sun)_1$ in $X^\sun$, we obtain that
     $
        \Phi_1' = \mirrorR_1\Psi_1
     $
     on $X^\sun$.

    Next, fix $v\in \Ll^\infty([0,1];U)$. The operator $\Phi_{1,v} \colon \Ll^\infty([0,1]) \to X$ defined by $f \mapsto \Phi_1(fv)$
    is weakly compact by Lemma~\ref{lem:weak-compactness-input-operator}. By Gantmacher’s theorem \cite[Theorem~3.5.13]{Megginson1998}, the dual operator $\Phi_{1,v}'$ is also weakly compact. In particular, the restriction
    \[
        S^v \coloneqq \Phi_{1,v}'\restrict{X^\sun} \colon X^\sun \to \Ll^\infty([0,1])'
    \]
    is weakly compact, being the composition of
    $\Phi_{1,v}'$ with the bounded inclusion $X^\sun\hookrightarrow X'$. Since for each $f\in \Ll^\infty([0,1])$ it holds that
    \begin{align*}
            \Spr{S^v x^\sun}{f}_{\Ll^\infty([0,1])',\Ll^\infty([0,1])} & = \Spr{x^\sun  }{\Phi_{1,v}f}_{X',X}
                                         = \Spr{x^\sun}{\Phi_1(fv)  }_{X',X} \\
                                         &= \Spr{\Phi_1'x^\sun  }{fv}_{\Ll^\infty([0,1];U)',\Ll^\infty([0,1];U)} \\
                                        &= \int_0^1  \Spr{(\Phi_1'x^\sun)(s)}{f(s)v(s) }_{U',U}\dd s \\
                                         &= \int_0^1 f(s) \Spr{(\Phi_1'x^\sun)(s)}{v(s)}_{U',U}\dd s,
    \end{align*}
    and $\Spr{(\Phi_1'x^\sun)(\argument)}{v(\argument)}_{U',U}\in \Ll^1([0,1])$,
    it follows that $S^v x^\sun = \Spr{(\Phi_1'x^\sun)(\argument)}{v(\argument)}_{U',U}$ under the canonical embedding $\Ll^1([0,1])\hookrightarrow \Ll^\infty([0,1])'$.
    In particular,
    $S^v(X^\sun)\subseteq \Ll^1([0,1])$.

    Now, together, weak compactness of $S^v$ and the
    classical Dunford--Pettis theorem \cite[Theorem~III.2.15]{DiestelUhl1977} yield that the set $S^v(\ball_{X^\sun})\subseteq \Ll^1([0,1])$ is uniformly integrable.
    
    Consider the bounded set $M\coloneqq \Phi_1'(\ball_{X^\sun}) \subseteq \Ll^1([0,1];U')$ and set
    \[
        M_v \coloneqq \{ \Spr{f(\argument)}{v(\argument)}_{U',U}: f\in M\}.
    \]
    Then $M_v = S^v(\ball_{X^\sun})$ is uniformly integrable for all $v\in \Ll^\infty([0,1];U)$. Hence $M$ is uniformly integrable by Lemma~\ref{lem:uniform-integrability}. Since $M = (\mirrorR_1 \Psi_1)(\ball_{X^\sun})$ and reflection is measure preserving, $\Psi_1(\ball_{X^\sun})$ is uniformly integrable. Consequently, the assertion holds due to \cite[Theorem~3.4]{preusslerImplicationsOfStructured2026}.
\end{proof}

\section{Proofs of the main results}
    \label{sec:proofs}

In this section, we give the proofs of the results stated in Section~\ref{sec:main-results}.

\begin{proof}[Proof of Theorem~\ref{thm:main}]
    The implication
    \ref{thm:main:itm:infty-admissibility}$\Rightarrow$\ref{thm:main:itm:orlicz-admissibility}
    follows from Proposition~\ref{prop:linfty_to_ef}.
    Furthermore, the implications
    \ref{thm:main:itm:iiss}$\Rightarrow$\ref{thm:main:itm:iss}
    $\Rightarrow$\ref{thm:main:itm:infty-admissibility},
    as well as
    \ref{thm:main:itm:orlicz-admissibility}$\Rightarrow$\ref{thm:main:itm:iiss}
    under exponential stability of \Ttt{},
    follow from \cite[Proposition~2.10 and Theorem~3.1]{Jacob_2018}.
\end{proof}

\begin{proof}[Proof of Corollary~\ref{cor:weiss}]
    Since $\Ll^\infty$ admissibility of the control operator implies $\Ee_F$-admissibility for a Young function $F$ (Theorem~\ref{thm:main}), the assertion  follows from \cite[Proposition~5]{jacobContinuitySolutionsParabolic2019}.
    See also the discussion in \cite{hosfeldInputtostateStabilityClasses2025}.
\end{proof}

\begin{proof}[Proof of Corollary~\ref{cor:perturbation}]
    As noted in Remark~\ref{rem:linfty-to-orlicz}(b), $\Ll^\infty$-admissibility of $B\in\LL(X,X_{-1})$ is automatically zero-class. The assertion now follows from \cite[Corollary~III.3.3]{engel_one-parameter_2000}.
\end{proof}

\begin{proof}[Proof of Corollary~\ref{cor:bilinear}]
    By Theorem~\ref{thm:main}, the linear system $(A,B)$
    is $\Ll^\infty$-integral ISS. In particular, the semigroup
    generated by $A$ is exponentially stable, so $(A,0)$
    is also $\Ll^\infty$-integral ISS. The assertion follows from
    \cite[Theorem~2.9]{HosfeldJacobSchwenninger2022},
    applied with $F(x,u)=ux$, $B_1=B$, and $B_2=0$.
\end{proof}

\section{Counterexamples}

In analogy to admissibility with respect to $\Ll^\infty$ and $\Ee_F$, one can define the notion of (zero-class) $\mathrm C$-admissibility.
As a capstone to the paper, we provide two counterexamples, establishing 
\[
    \mathrm{C}\text{-admissibility} \not\Rightarrow \text{zero-class }\mathrm{C}\text{-admissibility} \not\Rightarrow \Ee_F\text{-admissibility}.
\]
Both examples use the scalar input space $U = \C$. Compared to the Kato example \cite[Example~2.3]{JacobSchwenningerWintermayr2022} where $\dim U = \infty$ is used to show the first non-implication above, the first example therefore provides a simplified setup. The second example is built on the Cantor function and the properties of its derivative and, to the authors' knowledge, provides the first example of 
a zero-class $\Cc$-admissible control operator that is not already $\Ee_F$-admissible for any Young function.

Combined, the constructions
show that in the realm of $\Cc$-admissibility, all three notions (admissibility, zero-class admissibility, and $\Ee_F$-admissibility for some $F$) are
distinct. With the results in Section \ref{sec:admissibility} of this paper, we see that the situation for $\Ll^\infty$-admissibility stands in great contrast to this,
since in that setting, the three admissibility properties coincide. This is a result of the additional weak compactness properties that follow from the 
Grothendieck structure of $\Ll^\infty([0,1])$. Note that the apparent detour through the ``scalarized world'' of Lemma \ref{lem:uniform-integrability} is necessary for
our approach since $\Ll^\infty([0,1];U)$ is itself not necessarily a Grothendieck space for $U$ arbitrary, see, e.g., \cite[Proposition~5.4.5]{gonzalez_grothendieck_2021}.

\begin{example}[$\mathrm{C}$-admissibility does not imply zero-class $\mathrm{C}$-admissibility, even if $U = \C$]
On $X = c_0$, consider the semigroup $(T(t))_{t\ge 0}$ generated by $A = \operatorname{diag}(-n + \ii n^2)$ and consider the input element $b = (n)_{n\in\N} \in X_{-1}$.
The dual system is $X^\sun = X' = \ell^1$ and $A' = \operatorname{diag}(-n + \ii n^2)$ with the observation operator  $C : (X')_1\to \C$ given by $z \mapsto \sum_{n=1}^\infty n z_n$. 

Applying \cite[Theorem~3.1]{aroraAdmissibleOperatorsSundual2026}, $b$ is $\mathrm{C}$-admissible if and only if $C$ is $\Ll^1$-admissible; suitably, the same equivalence also applies to the zero-class variants. For $z\in\ell^1$ and $t>0$, we have $T(t)'z\in \dom{A'}$,
since $\sup_{n\in\N}|-n+\ii n^2|\e^{-nt}<\infty$.
Moreover, Tonelli's theorem gives
\[
    \int_0^1\sum_{n=1}^\infty ne^{-nt}|z_n|\dd t
    =\sum_{n=1}^\infty|z_n|(1-e^{-n})
    \leq\norm{z}_{\ell^1}.
\]
Consequently, the output is integrable and the following
estimate is justified:
\begin{align*}
    \norm{CT(\argument)'z}_{\Ll^1([0,1])} &= \int_0^1 \biggl|{\sum_{n=1}^\infty}n \e^{-nt}\e^{\ii n^2 t}z_n\biggr| \dd t
    \le \sum_{n=1}^\infty \abs{z_n} \int_0^1 \abs{n \e^{-nt}\e^{\ii n^2 t}} \dd t\\
    &= \sum_{n=1}^\infty \abs{z_n} \int_0^1 n\e^{-nt} \dd t 
    = \sum_{n=1}^\infty \abs{z_n} (1 - \e^{-n}) \leq \norm{z}_{\ell^1};
\end{align*}
 see also
\cite[Example~4.3]{preusslerImplicationsOfStructured2026}.
It follows that $C$ is $\Ll^1$-admissible. Setting $z = e_k$, the $k$\textsuperscript{th} unit
vector in $\ell^1$, we see that for each $\epsilon >0$
\[
    \norm{CT(\argument)'e_k}_{\Ll^1([0,\epsilon])}  = \int_0^\epsilon k \e^{-kt}\dd t = 1-\e^{-k\epsilon} \to 1
\]
as $k \to \infty$. Thus the norm of the observation map
$z\mapsto CT(\argument)'z$ from $\ell^1$ to
$\Ll^1([0,\epsilon])$ is at least $1$ for every $\epsilon>0$.
Whence $C$, and in turn $b$, is admissible but not zero-class admissible.
\end{example}

\begin{example}[Zero-class $\mathrm{C}$-admissibility does not imply $\Ee_F$-admissibility, even if $U = \C$]
    Let $(L(t))_{t\ge 0}$ be the nilpotent left translation semigroup with generator $A$ on
    \[
        X \coloneqq \{ f\in \Cc[0,1]: f(1)=0\}.
    \]
    Let $\mathfrak c\colon[0,1]\to[0,1]$ denote the Cantor
    function and let $\mu$ be the associated Cantor measure.
    The function $\mathfrak c$ is continuous and nondecreasing,
    with $\mathfrak c(1)=1$, and $\mu$ is singular with respect
    to Lebesgue measure; see
    \cite[Examples~1.67 and~3.34]{ambrosioFunctionsBoundedVariation2000}.
    Consequently, $(\mathfrak c-1)\in X\cap\operatorname{BV}([0,1])$.
    Since $\mathfrak c$ is the distribution function of $\mu$,
    $\partial(\mathfrak c-1)=\mu$
    in $\mathcal D'((0,1))$;
    see \cite[Example~1.75]{ambrosioFunctionsBoundedVariation2000}.
    
    Therefore, $b \coloneqq A_{-1}(\mathfrak c-1) =\mu$ in $X_{-1}$; see \cite[Example~5.1]{BatkaiJacobVoigtWintermayr2018}. The same example identifies the positive extrapolation cone
    with the positive atomless finite measures on $(0,1)$.
    Since the Cantor measure is positive and atomless, $b$ is a positive operator when viewed as an element of $\mathcal L(\C, X_{-1})$. Consequently, $b$ is zero-class $\mathrm{C}$-admissible by \cite[Corollary~4.7]{AroraGlueckPaunonenSchwenninger2024}. 

    Let $F$ be a Young function and assume towards a contradiction that $b$ is $\Ee_F$-admissible. Let $\Phi_1$ denote the corresponding input operator and $\mirrorR_1$ be the reflection operator $f \mapsto f(1-\argument)$. Since $\mirrorR_1$ is an isometry on $\Ee_F([0,1])$ and $\delta_0\in X'$, we obtain that
    \[
        \psi \coloneqq \delta_0 \Phi_1 \mirrorR_1 \in \Ee_F([0,1])' \simeq \Ll_{\widetilde F}([0,1]) \subseteq \Ll^1([0,1]);
    \]
    where $\widetilde F$ denotes the complementary Young function corresponding to $F$. For the duality $\Ee_F([0,1])' \simeq \Ll_{\widetilde F}([0,1])$ used above, see, e.g., \cite[Theorem~4.13.6 and Remark~4.13.8]{pickFunctionSpaces2013}. 
    In particular, there exists $g\in \Ll^1([0,1])$ such that $\Spr{\psi}{u}=\int_0^1 ug\dd s$ for all $u\in \Ee_F([0,1])$.

    On the other hand, for $\varphi\in\mathrm C^\infty_{\mathrm c}((0,1))$, integration by parts gives
    \begin{align*}
        \Phi_1 \mirrorR_1 \varphi & = \int_0^1 L_{-1}(s)A_{-1}[\mathfrak c-1]\varphi(s)\dd s
                                   = \int_0^1 \frac{\dd}{\dd s}\big( L_{-1}(s)[\mathfrak c-1]\big) \varphi(s)\dd s\\
                                  & = - \int_0^1 L(s)[\mathfrak c-1]\varphi'(s)\dd s,
    \end{align*}        
    where the last integral is an $X$-valued Bochner integral.
    In particular, evaluating the above expression at zero, we obtain, for every
    $\varphi\in\Cc^\infty_{\mathrm c}((0,1))$,
    \begin{align*}
        \psi(\varphi) & = \left[-\int_0^1 L(s)[\mathfrak c-1]\varphi'(s)\dd s  \right]\!(0)
              = -\int_0^1 \big(L(s)[\mathfrak c-1]\big)(0) \varphi'(s)\dd s\\
             & = -\int_0^1 [\mathfrak c-1](s)\varphi'(s)\dd s
              = \int_0^1 \varphi(s)\dd \mu(s).
    \end{align*}
    On the other hand, $\Spr{\psi}{\varphi}=\int_0^1\varphi(s)g(s)\dd s$.
    Since $\varphi$ was arbitrary, $\mu$ agrees with $g(s)\dd s$ as a measure
    on $(0,1)$. This is impossible: $\mu((0,1))=1$, while $\mu$ is singular
    with respect to Lebesgue measure. Thus $b$ is not $\Ee_F$-admissible for
    any Young function $F$.
\end{example}

\printbibliography

\end{document}